\documentclass[11pt,oneside]{amsart}
\usepackage{amssymb,latexsym}
\usepackage{amsmath}
\usepackage{amsfonts}
\usepackage{multicol}
\usepackage[height=188mm,width=129mm]{geometry}
\theoremstyle{plain}
\newtheorem{proposition}{Proposition}
\newtheorem*{corollary}{Corollary}

\theoremstyle{remark}

\begin{document}
\title[On the polyhedral cones of convex and concave vectors]{On the polyhedral cones\\ of  convex and concave vectors}
\author{Stephan Foldes}
\address{Stephan Foldes \newline%
\indent Institute of Mathematics,   \newline%
\indent Tampere University of Technology,  \newline%
\indent PL 553, 33101 Tampere, Finland }
\email {sf@tut.fi}%

\author{L\'aszl\'o Major}
\address{L\'aszl\'o Major\newline%
\indent Institute of Mathematics,   \newline%
\indent Tampere University of Technology,  \newline%
\indent PL 553, 33101 Tampere, Finland }\email{laszlo.major@tut.fi}%
\hspace{-4mm} \date{Nov 18, 2013}

\begin{abstract}Convex or concave sequences of $n$ positive terms, viewed as vectors in $n$-space, constitute convex cones with $2n-2$ and $n$ extreme rays, respectively. Explicit description is given of vectors spanning these extreme rays, as well as of non-singular linear transformations between the positive orthant and the simplicial cones formed by the positive concave vectors. The simplicial cones of monotone convex and concave vectors can be described similarly.
\end{abstract}
\maketitle
In this note a sequence (vector) $\mathbf{a}=(a_1,\ldots, a_n)$, $n\geq 1$, of
 real numbers is called \textit{positive} if for all $1\leq i \leq n$, $0\leq a_i$, \textit{increasing} if for all $1\leq i < n$, $0\leq a_{i+1}-a_i$ and \textit{convex} if for all $1< i < n$, $0\leq a_{i+1}-2a_i+a_{i-1}$. The sequence $\textbf{a}$ is \textit{negative}, \textit{decreasing}, or \textit{concave}, when $-\textbf{a}$ is positive, increasing, or convex, respectively. The vector $\textbf{a}$ is \textit{unimodal} if it is the concatenation of an increasing and a decreasing sequence.  All increasing, decreasing and concave vectors are unimodal. The question of unimodality of the members of a class of sequences naturally arising in combinatorics can be difficult (e.g. Whitney numbers \cite{D1, R1} and face vectors of certain classes of polytopes \cite{SZ1,M1}). Proof of unimodality of $\textbf{a}=a_1,\ldots,a_n$ (where all $a_i>0$) is sometimes based on proving the stronger property of concavity of $\log a_1,\ldots,\log a_n$ (log-concavity) \cite{B1,D1,S1,M1}. In turn, to prove that log-concavity is preserved in certain constructions of sequences, ordinary concavity of some coefficient sequences may be used \cite{K1}. Each of the sets of positive, negative, increasing, decreasing, convex and concave vectors, and various intersections of these sets, form closed \textit{cones} (sets containing the null vector and closed under linear combinations with non-negative coefficients, called \textit{conic combinations}). The cone of positive vectors  is the positive orthant in $\mathbb{R}^n$. For positive increasing vectors this cone was described by Lov\'asz (\cite{L1}, p. 248, last equation) and by Marichal and Mathonet \cite{MM1} in terms of its intersection with the unit hypercube (one of the $n!$ simplices of the standard triangulation of the hypercube). In this note we describe the cone of positive concave and positive convex vectors by determining the extreme rays of the cones that they form. When this cone is simplicial we describe a standardized matrix realizing the transformation of the orthant to the cone in question.

Within $\mathbb{R}^n$, let $C$ be the set of positive concave vectors having maximal component value $1$.  In $C$ there are exactly $n$ minimal vectors with respect to the componentwise order, called \textit{minimal standard concave vectors}.  We shall denote these $\textbf{c}^{(1)},\ldots,\textbf{c}^{(n)}:$ $\textbf{c}^{(i)}$ being the minimal concave vector whose  $i^{th}$ component is maximal. For $1<i<n$, ${c}_1^{(i)}=0, {c}_2^{(i)},\ldots, {c}_i^{(i)}=1$ is an increasing arithmetic progression and ${c}_i^{(i)}=1, {c}_{i+1}^{(i)},\ldots, {c}_n^{(i)}=0$ is a decreasing arithmetic progression, while $\textbf{c}^{(1)}$ and $\textbf{c}^{(n)}$ are decreasing, respectively increasing arithmetic progressions from $0$ to $1$ or $1$ to $0$.
\begin{proposition}The $n$ minimal standard concave vectors $\textbf{c}^{(1)},\ldots,\textbf{c}^{(n)}$ form a basis of $\mathbb{R}^n$. Every positive concave vector $\textbf{c}$ can be written uniquely as 
$$\textbf{c}=\lambda_1\textbf{c}^{(1)}+\cdots +\lambda_n\textbf{c}^{(n)}$$ 
with $\lambda_1\ldots, \lambda_n\geq 0.$
\end{proposition}
\begin{proof}We prove only the second statement which implies the first. For a given positive concave vector $\textbf{c}=(c_1,\ldots,c_n)$, $\textbf{c}-c_1\textbf{c}^{(1)}-c_n\textbf{c}^{(n)}$ is still positive concave vector and has $0$ as its first and last component. Thus it is enough to prove that any positive concave vector $\textbf{c}=(c_1,\ldots,c_n)$ with $c_1=c_n=0$ is a unique conic combination of $\textbf{c}_2,\ldots,\textbf{c}_{n-1}$. Calling an index $i$, $1<i<n$, \textit{singular} if $2c_i>c_{i-1}+c_{i+1}$, we use induction on the number of singular indices. This number is $0$ if and only if $\textbf{c}$ is the null vector, in which case the assertion is obvious. Otherwise let $i$ be the first singular index in $\textbf{c}$. There is a unique positive real number $\lambda_i$ such that $i$ is not a singular index in $\textbf{c}-\lambda_i \textbf{c}^{(i)}$. Every index $j\neq i$ is singular in $\textbf{c}$ if and only if it is singular in $\textbf{c}-\lambda_i \textbf{c}^{(i)}$. Applying the induction hypothesis to $\textbf{c}-\lambda_i \textbf{c}^{(i)}$ completes the proof.
\end{proof}
%
It follows from the above that the cone of positive concave vectors is the image of the positive orthant under the (non-singular) linear transformation represented by the matrix $M$ whose rows are the minimal standard concave vectors $\textbf{c}_1,\ldots,\textbf{c}_n$, which span the extreme rays of the cone.  The matrix $M$ is centrally symmetric, i.e. $\textbf{c}_i$ is the reverse sequence of $\textbf{c}_{n-i+1}$ (or equivalently $M(i,j)=M(n-i+1,n-j+1)$). In fact $M$ can be defined as the only centrally symmetric matrix with main diagonal constant $1$, for which the entries $M(i,j)$ under the main diagonal $(i>j)$ are given by $M(i,j)=\frac {j-1}{ i-1}$. The linear transformation mapping the positive orthant to the positive concave vectors is then given by  $\lambda\mapsto\lambda M$.\\
%
The inverse of $M$ (also centrally symmetric) is the $n\times n$ matrix whose only non-zero entries are given by 
\renewcommand{\theenumi}{\roman{enumi}}
\begin{enumerate}
  \item the main diagonal\vspace{2mm}  $1,\ldots,\frac{2(i-1)(n-i)}{n-1},\ldots,1$,
  \item fo\vspace{2mm}r $1<j<n$ the entries  $M^{-1}(j-1,j)=M^{-1}(j+1,j)=-\frac 12 M^{-1}(j,j)$
  \end{enumerate}
E.g.  for $n=5$, we have
\begin{multicols}{2}\begin{center}
$M=\displaystyle\frac{1}{12}\cdot\begin{bmatrix}
12&9&6&3&0\\
0&12&8&4&0\\
0&6&12&6&0\\
0&4&8&12&0\\
0&3&6&9&12\\
\end{bmatrix}$
\end{center}
%
$M^{-1}=\displaystyle\frac{1}{12}\cdot\begin{bmatrix}
12&-9&0&0&0\\
0&18&-12&0&0\\
0&-9&24&-9&0\\
0&0&-12&18&0\\
0&0&0&-9&12\\
\end{bmatrix}$
\end{multicols}
$M^{-1}$ transforms the cone of positive concave vectors to the positive orthant. All non-zero entries of $M^{-1}$ are on three diagonals. In comparison, the cone of increasing positive vectors is transformed (following from \cite{L1}) to the positive orthant a by the  matrix $Z^{-1}$, whose non-zero entries are on two diagonals. Both $Z^{-1}$ and $M^{-1}$ have column sums equal to $0$, except for the first column of $Z^{-1}$ and the first and last column of $M^{-1}$.\\

For $1\leq i < n$ let $C_i$ (respectively $D_i$) be the set of those positive increasing (resp. decreasing) convex vectors with maximal component value $1$ that have exactly $i$ components equal to $0$.
Then $C_i$ (resp. $D_i$) has a unique maximal vector $\textbf{a}_i$ (resp. $\textbf{b}_i$) with respect to the componentwise ordering, called the $i^{th}$ \textit{standard increasing (resp. decreasing) convex vector}. Note $\textbf{1}>\textbf{a}^{(1)}>\cdots > \textbf{a}^{(n-1)}$ and $\textbf{1}>\textbf{b}^{(1)}>\cdots > \textbf{b}^{(n-1)}$, where $\textbf{1}=(1,\ldots,1)$.
\begin{proposition}
\upshape \textit{The} $n-1$ \textit{standard increasing convex vectors are linearly independent and together with} $\textbf{1}$ \textit{form a basis of $\mathbb{R}^n$. Every positive increasing convex vector} $\textbf{c}$ \textit{can be written uniquely as} 
$$\textbf{c}=\lambda_1\textbf{a}^{(1)}+\cdots +\lambda_{n-1}\textbf{a}^{(n-1)}+\lambda_n \textbf{1}$$ 
\textit{with} $\lambda_1\ldots, \lambda_n\geq 0.$
\end{proposition}
\begin{proof}The first statement is obvious. It is enough to prove the second statement for $\textbf{c}=(c_1,\ldots,c_n)$ with $c_1=0$ (because $c_1\textbf{1}$ can be subtracted). In such a vector call an index $1<i<n$ \textit{singular}  if $2a_i>a_{i-1}+a_{i+1}$. Now we use induction on the number of singular indices. This number is $0$ only in the obvious case where $\textbf{c}$ is the null vector. Otherwise let $i$ be the first singular index in $\textbf{c}$. There is a unique positive real number $\lambda_i$ such that $i$ is not a singular index in $\textbf{c}-\lambda_i \textbf{a}^{(i)}$ and to this vector we can apply the induction hypothesis.
\end{proof}
\begin{corollary}\upshape
\textit{The} $n-1$ \textit{standard decreasing convex vectors are linearly independent and with} $\textbf{1}$ \textit{form a basis of $\mathbb{R}^n$. Every positive decreasing convex vector} $\textbf{c}$ \textit{can be written uniquely as }
$$\textbf{c}=\lambda_1\textbf{b}^{(1)}+\cdots +\lambda_{n-1}\textbf{b}^{(n-1)}+\lambda_n \textbf{1}$$ 
\textit{with} $\lambda_1\ldots, \lambda_n\geq 0.$
\end{corollary}
\begin{proposition}In $\mathbb{R}^n$ the cone of positive convex vectors has $2n-2$ extreme rays spanned  by the standard increasing and standard decreasing convex vectors.  
\end{proposition}
\begin{proof}Every positive convex vector can be written (not uniquely) as the sum of an increasing and a decreasing positive convex vector. The vector $\textbf{1}$ equals $\textbf{a}^{(1)}+\textbf{b}^{(n)}$. Further, since none of the $\textbf{a}^{(i)}$ can be a conic combination of the other vectors $\textbf{a}^{(j)}$ and the various $\textbf{b}^{(k)}$, and similarly none of the $\textbf{b}^{(i)}$ is a combination of the other vectors, all the rays generated by the standard convex vectors are extremal.
\end{proof}
%
Obviously the representation of a positive convex vector as a conic combination of standard convex vectors  is not unique. However, every positive convex vector $\textbf{c}$ in $\mathbb{R}^n$ can be written with coefficients $\lambda_1\ldots, \lambda_{n-1}$ and $\theta_1\ldots, \theta_{n-1}$, uniquely determined by $\textbf{c}$, in the form
$$\textbf{c}=(\min \textbf{c})\textbf{1}+\sum \lambda_i \textbf{a}^{(i)}+\sum \theta_i \textbf{b}^{(i)}$$ where the vectors $\textbf{a}^{(i)}$ and $\textbf{b}^{(i)}$ are the standard convex vectors. Not all possible combinations of coefficients $\lambda_i$,  $\theta_i$ can appear in such a representation.

The cone of positive increasing convex vectors is the image of the positive orthant  under the (non-singular) linear transformation represented by the matrix $N$ whose rows are the vector $\textbf{1}$ and the standard increasing convex vectors $\textbf{a}^{(1)},\ldots,\textbf{a}^{(n-1)}$.  The matrix $N$ is an upper triangular matrix, whose first row is the vector $\textbf{1}$ and for all $1<i\leq j \leq n$,  $N(i,j)=\frac {j-i+1}{n-i+1}$.

The inverse matrix $N^{-1}$   is the  $n\times n$ upper triangular matrix for which $N^{-1}(i,i)=(N(i,i))^{-1}$ ($1\leq i \leq n$), $N^{-1}(i,j)\neq 0$ for $0\leq j-i \leq 2$ only,  and all column sums and row sums are equal to $0$, except for the first column and the last row.
E.g.  for $n=5$, we have
\begin{multicols}{2}\begin{center}
$N=\displaystyle\frac{1}{12}\cdot\begin{bmatrix}
12&12&12&12&12\\
0&3&6&9&12\\
0&0&4&8&12\\
0&0&0&6&12\\
0&0&0&0&12\\
\end{bmatrix}$
\end{center}
%
$N^{-1}=\displaystyle\begin{bmatrix}
1&-4&3&0&0\\
0&4&-6&2&0\\
0&0&3&-4&1\\
0&0&0&2&-2\\
0&0&0&0&1\\
\end{bmatrix}$
\end{multicols}

The cone of positive decreasing convex vectors, its $n$ extreme rays, and the linear transformation matrix between that cone and the positive orthant has an entirely analogous description. The same can be done for the cones of positive increasing concave and positive decreasing  concave vectors. These cones are all simplicial cones, images of the positive orthant under a linear  transformation, whose inverse is represented by a matrix of special form, as the matrices  $M^{-1}$ and $N^{-1}$ above. This matrix is almost diagonal, in the sense that all non-zero entries are on two or three diagonals. Moreover the row sums and also the column sums of these inverse matrices are constant with the exception of a single special row and a single column. 

\end{document}